\documentclass[american, pdftex]{scrartcl}
\KOMAoptions{DIV=12, paper=a4, abstract=on}
\usepackage[utf8]{inputenc}
\usepackage[T1]{fontenc}
\usepackage{babel, csquotes, lmodern, microtype, fixltx2e}
\usepackage{amsmath, amssymb, amsthm, pgfplots, subcaption}

\usepackage[style=alphabetic,backend=bibtex8]{biblatex}
\AtEveryBibitem{%
   \ifentrytype{book}{\clearfield{pages}}{}
   \ifentrytype{article}{\clearfield{url}}{}
}
\renewbibmacro*{publisher+location+date}{%
  \printlist{publisher}%
  \setunit*{\addcomma\space}%
  \printlist{location}%
  \setunit*{\addcomma\space}%
  \usebibmacro{date}%
  \newunit{}
}

\usepackage{hyperref}
\hypersetup{pdfauthor={Lennard Kamenski and Weizhang Huang},
   pdftitle={Conditioning of FE equations / a density function approach},
   pdfsubject={Preprint}
}

\def\graphicsPath{.}
\def\dataPath{.}
\pgfplotsset{%
   compat=newest,
   width=1.0\textwidth, height=0.7\textwidth,
   legend cell align=left, legend pos=south west,
   ymajorgrids, yminorticks=false,
   xmajorgrids, xminorticks=false,
   every axis/.append style={font=\footnotesize, mark options={solid},
      line width=0.7pt},
}
\newcommand{\MyLineStyle}{solid}
\newcommand{\MyLineStyleScaled}{solid}
\newcommand{\MyLineStyleUniform}{dashed}
\newcommand{\MyColorExact}{black}
\newcommand{\MyColorUniform}{black}

\newcommand{\MyColorKHX}{blue}
\newcommand{\MyColorRho}{green!50!black}
\newcommand{\MyMarkExact}{|}
\newcommand{\MyMarkKHX}{square*}
\newcommand{\MyMarkRho}{triangle*}

\makeatletter \g@addto@macro\@floatboxreset\centering \makeatother

\providecommand{\bu}{\boldsymbol{u}}
\providecommand{\bx}{\boldsymbol{x}}
\providecommand{\bxi}{\boldsymbol{\xi}}
\DeclareMathOperator{\Diag}{diag}
\providecommand{\diag}{\Diag}
\DeclareMathOperator{\spanM}{span}
\DeclareMathOperator{\tr}{tr}
\providecommand{\abs}[1]{\lvert#1\rvert}
\providecommand{\Abs}[1]{\left\lvert#1\right\rvert}
\providecommand{\Norm}[1]{\left\lVert#1\right\rVert}
\providecommand{\dx}{\,d\boldsymbol{x}}
\providecommand{\dxi}{\,d\boldsymbol{\xi}}
\providecommand{\FDF}{ { (F_K') }^{-1} \D_K { (F_K') }^{-T}}
\providecommand{\SAS}{S^{-1} A S^{-1}}
\newcommand{\D}{\mathbb{D}}
\newcommand{\cO}{\mathcal{O}}
\newcommand{\R}{\mathbb{R}}
\newcommand{\cT}{\mathcal{T}}
\newcommand{\V}[1]{\boldsymbol{#1}}
\newcommand{\Krho}{\Abs{K}_\rho}
\newcommand{\KrhoMin}{\Abs{K_{\min}}_\rho}
\newcommand{\Wrho}{\Abs{\omega_j}_\rho}
\newcommand{\WrhoMin}{\Abs{\omega_{\min}}_\rho}

\newtheorem{theorem}{Theorem}[section]
\newtheorem{lemma}{Lemma}[section]
\theoremstyle{definition}
\newtheorem{rem}{Remark}[section]

\begin{document}


\title{A study on the conditioning of finite element equations
   with~arbitrary anisotropic meshes via~a~density function approach}

\author{Lennard Kamenski%
   \thanks{%
      Weierstrass Institute for Applied Analysis and Stochastics,
      Berlin, Germany
      (\href{mailto:kamenski@wias-berlin.de}{\nolinkurl{kamenski@wias-berlin.de}}).
      \newline{}
      Supported in part by the DFG under grant KA~3215/2-1.
   }
\and Weizhang Huang%
   \thanks{%
      Department of Mathematics, the University of Kansas,
      Lawrence, KS, 
      USA
      (\href{mailto:huang@math.ku.edu}{\nolinkurl{huang@math.ku.edu}}).
      \newline{}
      Supported in part by the NSF under grant DMS-1115118.
   }
}

\maketitle

\begin{abstract}
The linear finite element approximation of a general linear diffusion problem with arbitrary anisotropic meshes is considered.
The conditioning of the resultant stiffness matrix and the Jacobi preconditioned stiffness matrix is investigated using a density function approach proposed by Fried in~1973.
It is shown that the approach can be made mathematically rigorous for general domains and used to develop bounds on the smallest eigenvalue and the condition number that are sharper than existing estimates in one and two dimensions and  comparable in three and higher dimensions.
The new results reveal that the mesh concentration near the boundary has less influence on the condition number than the mesh concentration in the interior of the domain.
This is especially true for the Jacobi preconditioned system where the former has little or almost no influence on the condition number.
Numerical examples are presented.

\begin{description}
\item[Keywords:] conditioning, finite element, anisotropic diffusion, anisotropic mesh, stiffness matrix, extreme eigenvalue, Jacobi preconditioning, diagonal scaling.
\item[2010 MSC:] 65N30, 65N50, 65F35, 65F15.
\end{description}

\end{abstract}


\section{Introduction}

Mesh adaptation is a common tool for use in the numerical solution of partial differential equations (PDEs) to enhance computational efficiency.
It often results in nonuniform meshes whose elements vary significantly in size and shape from place to place on the physical domain.
Nonuniform meshes could lead to ill-conditioned linear systems and their solution may deteriorate the efficiency of the entire computation.
It is thus important in practice as well as in theory to understand how mesh nonuniformity affects the conditioning of linear systems resulting from discretization of PDEs on nonuniform meshes.

The issue has been studied by a number of researchers mostly for the linear finite element approximation of the Laplace operator or a general diffusion operator by developing bounds on the extremal eigenvalues on the resultant stiffness matrix, e.g.,\ see~\cite{Fri73, Wat87, BanSco89, Shewch02, DuWanZhu09, KamHuaXu13} for second-order elliptic PDEs or~\cite{AinMcLTra99, AinMcLTra00, GraMcL06} for a more general setting of elliptic bilinear forms on Sobolev spaces of real index $m \in [-1,1]$.

The estimation of the largest eigenvalue is well understood and it is easy to show that the largest eigenvalue is bounded by a multiple (with a constant depending on mesh connectivity) of the maximum of the largest eigenvalues of the local stiffness matrices~\cite{Fri73}.
Moreover, the largest diagonal entry of the stiffness matrix is a good estimate for the largest eigenvalue: it is tight within a factor of at most $d+1$ for any mesh, where $d$ is the dimension of the physical domain~\cite{KamHuaXu13}.
Sharp estimates in terms of mesh geometry are available for both isotropic~\cite{AinMcLTra99, Shewch02, DuWanZhu09, GraMcL06} and anisotropic~\cite{KamHuaXu13} diffusion.

The estimation of the smallest eigenvalue is more challenging.
Currently there are two approaches for this purpose.
The first approach utilizes Sobolev's inequality and is first used by Bank and Scott~\cite{BanSco89} for the Laplace operator with isotropic meshes in $d \geq 2$ dimensions to develop a lower bound on the smallest eigenvalue of a diagonally scaled stiffness matrix and show that the condition number of the scaled stiffness matrix is comparable to that with a uniform mesh.
A similar result for elliptic bilinear forms on Sobolev spaces of real index $m \in [-1,1]$ with shape-regular meshes in $d \geq 2$ dimensions was derived by Ainsworth, McLean, and Tran~\cite{AinMcLTra99,AinMcLTra00}.
Their result was later generalized to locally quasi-uniform meshes\footnote{Meshes with neighboring elements having comparable size and shape.} in $d \geq 3$ dimensions by Graham and McLean~\cite{GraMcL06}.
Recently, Kamenski, Huang, and Xu~\cite{KamHuaXu13} derived a similar bound for second-order elliptic PDEs which is valid for arbitrary meshes (i.e.,\ without imposing any conditions on the mesh regularity) for any $d$;
the established bound for the condition number depends on three factors, one representing the condition number of the linear finite element equations for the Laplace operator on a uniform mesh and the other factors arising from the nonuniformity of the mesh viewed in the metric specified by the inverse of the diffusion matrix $\D$ (\emph{$\D^{-1}$-nonuniformity}) and the mesh nonuniformity in volume measured in the Euclidean metric (\emph{volume-nonuniformity}).
Further, it was shown in~\cite{KamHuaXu13} that the Jacobi preconditioning --- an optimal diagonal scaling for a symmetric positive definite sparse matrix --- eliminates the effect of the mesh volume-nonuniformity and reduces the effect of the mesh $\D^{-1}$-nonuniformity.
This result can be seen as a further generalization of~\cite{BanSco89,AinMcLTra99,GraMcL06} towards arbitrary anisotropic meshes and general diffusion coefficients.

In the second approach (hereafter referred to as \emph{the density function approach}), a lower bound on the smallest eigenvalue of the stiffness matrix is obtained through a lower bound of the smallest eigenvalue of a continuous generalized eigenvalue problem involving an auxiliary density function.
It is first employed by Fried~\cite{Fri73} for the Laplace operator.
Fried shows how to obtain the lower bound for the continuous problem for spherical domains and claims via a physical argument that the bound is also valid for general domains.
The obtained bound for the smallest eigenvalue of the stiffness matrix is valid for any mesh in any dimension but in three and higher dimensions it is less sharp than bounds obtained in~\cite{GraMcL06,KamHuaXu13}: Fried's bound is proportional to power $(d-2)/d$ of the volume of the smallest mesh element whereas those in~\cite{GraMcL06,KamHuaXu13} are proportional to an average of the volumes of all elements (cf.~\eqref{eq:KamHuaXu13} and~\eqref{eq:lMin:A-Fried}).

In this paper we investigate the density function approach to develop bounds on the smallest eigenvalue and the condition number of both the stiffness matrix and the Jacobi preconditioned stiffness matrix for the linear finite element approximation of a general diffusion problem with arbitrary nonuniform meshes.
We show that the approach yields a bound on the smallest eigenvalue that is much sharper than in the original work by Fried~\cite{Fri73}.
Moreover, the new results are even sharper than those obtained in~\cite{KamHuaXu13} in one and two dimensions and comparable in three and higher dimensions. 
In particular, they involve a factor describing the distance of a given element to the boundary and reflect the fact that the mesh concentration near the boundary has less influence on the condition number than the mesh concentration in the interior of the domain.
This is especially true for the Jacobi preconditioned system where the former has little or almost no influence on the condition number.

The outline of the paper is as follows.
The linear finite element approximation of the general diffusion problem is described in Sect.~\ref{sect:FE}.
The lower bound for the smallest eigenvalue of the stiffness matrix is developed in Sect.~\ref{sect:lambdaMinRho}, followed by the derivation of the bounds on the condition number in Sect.~\ref{sect:condition}.
Section~\ref{sect:condition} also contains comparison of the obtained bounds to those in~\cite{Fri73} and~\cite{KamHuaXu13}.
Numerical examples are presented in Sect.~\ref{sect:num}.
The conclusions are drawn in Sect.~\ref{SEC:conclusion}.

\section{Finite element approximation}\label{sect:FE}

We consider the boundary value problem (BVP) of a prototype anisotropic diffusion equation in the form
\begin{equation}
   \begin{cases}
      - \nabla \cdot \left( \D \nabla u \right) = f, & \text{in $\Omega$}, \\
      u = 0,                            & \text{on $\partial\Omega$},
   \end{cases}
   \label{eq:bvp1}
\end{equation}
where $\Omega$ is a simply connected polygonal or polyhedral domain in $\R^d$ ($d \geq 1$) and $\D=\D(\bx)$ is the diffusion matrix.
We assume that $\D$ is symmetric and positive definite and there exist two positive constants $d_{\min}$ and $d_{\max}$ such that
\begin{equation}
   d_{\min} I \leq \D(\bx) \leq d_{\max} I, \qquad \forall \bx \in \Omega, 
   \label{eq:D:1}
\end{equation}
where $I$ is the $d\times d$ identity matrix and the less-than-or-equal sign means that the difference between the right-hand side and left-hand side terms is positive semidefinite.
The weak formulation of the BVP~\eqref{eq:bvp1} is to find $u \in V \equiv H_0^1(\Omega)$ such that
\begin{equation}
   \mathcal{A}(u,v) = \mathcal{F}(v), \qquad \forall v \in V ,
   \label{eq:vp}
\end{equation}
where
\[
   \mathcal{A}(u,v) =  \int_{\Omega} \nabla v \cdot \D \nabla u \dx
   \qquad \text{and} \qquad 
   \mathcal{F}(v) = \int_{\Omega} f v \dx .
\]

Let an affine family of simplicial meshes $\left\{\cT_h\right\}$ for $\Omega$ be given.
For any element $K \in \cT_h$, let $F_K \colon \hat{K} \to K$ be the affine mapping from the reference element $\hat{K}$ to the mesh element $K$ and $F_K'$ the Jacobian matrix of $F_K$.
For notational simplicity, we assume that $\abs{\hat{K}} = 1$.
Note that $F_K'$ is constant on $K$.

Denote the linear finite element space associated with mesh $\mathcal{T}_h$ by $V^h \subset V$.
Then, a linear finite element solution $u_h \in V^h$ of BVP~\eqref{eq:bvp1} is defined by
\begin{equation}
   \mathcal{A}(u_h,v_h) = \mathcal{F}(v_h), \qquad \forall v_h \in V^h .
   \label{eq:fem}
\end{equation}
Let $N$ and $N_{vi}$ be the numbers of the elements and the interior vertices of $\cT_h$, respectively.
If we order the vertices in such a way that the first $N_{vi}$ vertices are interior vertices and denote by $\phi_j$ the standard linear basis function associated with the $j^{\text{th}}$ vertex, then we can express $V^h$ and $u_h$ as
\begin{equation}
   V^h = \spanM \lbrace \phi_1, \dotsc, \phi_{N_{vi}} \rbrace 
   \qquad \text{and} \qquad
   u_h = \sum\limits_{j=1}^{N_{vi}} u_j \phi_j  .
   \label{uh-1}
\end{equation}
In the following we will use the function form $u_h = \sum_{j} u_j \phi_j$ and the vector form $\bu = {[u_1,\dotsc, u_{N_{vi}}]}^T$ synonymously.
Using this, we can write~\eqref{eq:fem} in a matrix form as
\[
   A \bu = \V{f},
\]
where the stiffness matrix $A$ and the right-hand side term $\V{f}$ are given by
\begin{align}
   A_{ij} &= \mathcal{A}(\phi_j,\phi_i)  = \sum_{K \in \cT_h} \Abs{K} \nabla\phi_i \cdot \D_K \nabla \phi_j,
   &  i, j &= 1,\dotsc, N_{vi}  ,
   \label{eq:Aij} \\ 
   f_i &= \sum_{K \in \cT_h}  \int_{K} f \phi_i \dx, 
      &  i &= 1,\dotsc, N_{vi} ,
      \notag
\end{align}
and $\D_K$ is the average of $\D(\bx)$ over $K$, i.e., 
\[
   \D_K = \frac{1}{\Abs{K}} \int_K \D(\bx) \dx .
\]
We are interested in the condition numbers of matrices $A$ and $\SAS$, where $S = \sqrt{ \diag(A) }$ is the Jacobi preconditioner.

\section{Smallest eigenvalue of the stiffness matrix}\label{sect:lambdaMinRho}

In this section we develop lower bounds for $\lambda_{\min}(A)$ and $\lambda_{\min}(\SAS )$ using the density function approach proposed by Fried~\cite{Fri73}, where $S = \sqrt{\diag (A)}$ is the Jacobi preconditioner.
The approach utilizes an auxiliary eigenvalue problem
\begin{align}
   \begin{cases}
      -\Delta  u = \lambda \rho u,  & \text{in $\Omega$}, \\
      u = 0,                        &\text{on $\partial\Omega$},
   \end{cases}
   \label{CEVP}
\end{align}
where $\rho$ is a density function (to be chosen) satisfying
\begin{align}
   0 < \rho_{\min} \leq \rho \leq \rho_{\max} < \infty
   \qquad \text{and}
   \qquad \int\limits_{\Omega} \rho(\bx) \dx = 1 .
   \label{eq:rhoConditions}
\end{align}
The corresponding Galerkin formulation reads as
\[
   (\nabla u,\nabla v) = \lambda (\rho u, v),
   \qquad \forall v \in H_0^1(\Omega).
\]
Let $\lambda_\rho$ be the smallest eigenvalue of this eigenvalue problem and $\lambda_{\min}(B_\rho)$ the smallest eigenvalue of the Galerkin mass matrix $B_\rho$ associated with the density $\rho$.
Then, for any vector $\bu$, 
\begin{align*}
  \bu^T A \bu & = (\D \nabla u_h, \nabla u_h) \\
    &\ge d_{\min} (\nabla u_h, \nabla u_h) \\
    &\geq d_{\min} \lambda_\rho \left( \rho u_h, u_h \right)  \\
 &\geq d_{\min} \lambda_\rho \lambda_{\min}(B_\rho) \Norm{\bu}_2^2 ,
\end{align*}
which leads to
\begin{equation}
   \lambda_{\min}(A) \geq d_{\min} \lambda_\rho \lambda_{\min}(B_\rho).
   \label{eq:lminA:general}
\end{equation}
Thus, the key to the approach is to estimate $\lambda_\rho$ and choose the density function so that the lower bound~\eqref{eq:lminA:general} is as large as possible.

Hereafter, we assume that $\rho$ is piecewise constant, i.e.,  $\rho\vert_K = \rho_K = const.$ for all $K \in \cT_h$.
We denote
\[
   \Krho = \rho_K \Abs{K} ,
   \qquad \Wrho = \sum_{K \in \omega_j} \Krho ,
   \qquad \WrhoMin = \min\limits_{j=1, \dotsc,N_{vi}}  \Wrho .
\]
It is known~\cite[Sect.~3]{KamHuaXu13} that
\begin{equation}
   \lambda_{\min}(B_\rho)
   \geq \frac{\WrhoMin}{(d+1)(d+2)}.
   \label{eq:lminB}
\end{equation}
Inserting this into~\eqref{eq:lminA:general} yields
\begin{equation}
   \lambda_{\min}( A ) 
   \geq \frac{d_{\min} \WrhoMin}{(d+1)(d+2)} \lambda_\rho .
   \label{eq:lminA:general:2}
\end{equation}

\subsection{Estimation of the smallest eigenvalue of the continuous eigenvalue problem}

We now derive a lower bound on $\lambda_\rho$ by means of a Green's function.
The derivation consists of three lemmas, with the bound being given in Lemma~\ref{lem:lambda:min:Delta}.

First, we need the property of the strict positivity (or negativity) of eigenfunctions associated with the smallest eigenvalue of~\eqref{CEVP}.

\begin{lemma}
\label{lem:1}
For any density distribution $\rho$ satisfying~\eqref{eq:rhoConditions}, the smallest eigenvalue $\lambda_\rho$ of the eigenvalue problem~\eqref{CEVP} is simple and positive. Any corresponding eigenfunction is either strictly positive or strictly negative in $\Omega$.
\end{lemma}

\begin{proof}
The positiveness of $\lambda_\rho$ follows from
\[
   \lambda_\rho 
      = \min_{u \neq 0} \frac{(\nabla u,\nabla u)}{(\rho u, u)} 
      \geq \min_{u \neq 0} \frac{C_P \Norm{u}_2^2}{\rho_{\max} \Norm{u}_2^2}
      = \frac{C_P}{\rho_{\max}} > 0,
\]
where $C_P$ is the constant associated with the Poincar\'e's inequality.
The rest of the proof is the same as for the standard eigenvalue problem with $\rho \equiv 1$ (e.g.,\ see~\cite[Theorem~8.38]{GilTru01}), except that $\frac{(\nabla u,\nabla u)}{(u, u)}$ is replaced by $\frac{(\nabla u,\nabla u)}{(\rho u, u)}$.
\qquad
\end{proof}


\begin{lemma}
\label{lem:Green}
Let $G(\bx,\bxi)$ be the Green's function for $-\Delta$ subject to a homogeneous Dirichlet boundary condition on $\Omega$ and $d(\bx)$ the distance from $\bx$ to the boundary $\partial \Omega$, i.e., $d(\bx) = \min\limits_{\V{y} \in \partial \Omega} \Abs{\bx-\V{y}}$.
Then,
\begin{equation}
   0 \le G(\bx,\bxi) \leq C \times
      \begin{cases}
         \sqrt{d(\bx) d(\bxi)}, & \text{for $d = 1$}, \\
         \ln \left( 1 + \frac{ d(\bx) d(\bxi)}{\Abs{\bx-\bxi}^2 } \right),
            & \text{for $d = 2$}, \\
         \Abs{\bx-\bxi}^{2-d} \min \left\{ 1,  \frac{ d(\bx) d(\bxi)}{\Abs{\bx-\bxi}^2 } \right\},
            & \text{for $d\geq3$}.
      \end{cases}
   \label{eq:greens:function}
\end{equation}
\end{lemma}

\begin{proof}
For example, see~\cite{AgDoNi59, DalSwe04, GruWid82, IfrRia05}.
\qquad
\end{proof}


\begin{lemma}
\label{lem:lambda:min:Delta}
Assume that $\rho$ is piecewise constant and $p \in \left(1, \frac{d}{d-2}\right)$.
Let $d_K = \max_{\bx \in K} d(\bx)$.
Then the smallest eigenvalue $\lambda_\rho$ of the problem~\eqref{CEVP} is bounded from below by
\begin{equation}
   \lambda_\rho \geq C \times
      \begin{cases}
         {\left( \sum\limits_K \Krho d_K \right)}^{-1},
            & \text{for $d = 1$}, \\
         {\left( 1 + \sum\limits_{K} \Krho
            \ln^2 \left( 1 +  d_K \rho_{\max} \right) \right)}^{- \frac{1}{2}},
            & \text{for $d = 2$}, \\
        { {\left(\frac{d}{d-2} - p\right)}^{\frac{d}{d+2 p}} 
            \left( \sum\limits_K \Krho^{\frac{p}{p-1}} 
               \Abs{K}^{-\frac{1}{p-1}} d_K^{ \frac{p}{p-1} \cdot \frac{d-(d-2)p}{d + 2p}} \right)}^{-\frac{p-1}{p}},
         & \text{for $d\geq3$} . \\
      \end{cases}
   \label{eq:lambda:min:Delta}
\end{equation}
\end{lemma}

\begin{proof}
For a positive eigenfunction  $u_\rho$ associated with $\lambda_\rho$ we have
\begin{equation}
   u_\rho(\bx) = \lambda_\rho \int_\Omega G(\bx,\bxi) \rho(\bxi) u_\rho(\bxi) \dxi.
   \label{green-1}
\end{equation}

For $ 1\le d\leq 2$, applying the Cauchy-Schwarz inequality to~\eqref{green-1},
\[
   u_\rho(\bx)  \le \lambda_\rho {\left( \int_\Omega \rho(\bxi) {G(\bx,\bxi)}^2 \dxi \right)}^{\frac 1 2}
   {\left( \int_\Omega \rho(\bxi) {u_\rho(\bxi)}^2 \dxi \right)}^{\frac 1 2}.
\]
Taking square, multiplying with $\rho(\bx)$, and integrating both sides over $\Omega$ gives
\[
   \int_\Omega \rho(\bx) {u_\rho(\bx)}^2 \dx 
   \le \lambda_\rho^2 \int_\Omega  \rho(\bx) {\left[ \int_\Omega \rho(\bxi) {G(\bx,\bxi)}^2 \dxi \right]} \dx
   \cdot \int_\Omega \rho(\bxi) {u_\rho(\bxi)}^2 \dxi .
\]
This yields
\begin{equation}
   \lambda_\rho \ge {\left( \int_\Omega  \rho(\bx) {\left[ \int_\Omega \rho(\bxi) {G(\bx,\bxi)}^2 \dxi \right]} \dx \right)}^{-\frac 1 2} .
   \label{green-2}
\end{equation}

For $d=1$, using this and Lemma~\ref{lem:Green},
\begin{align*}
   \lambda_\rho
   \geq C
       {\left( \int_\Omega  \rho(\bx) {\left[ \int_\Omega \rho(\bxi) d(\bx)d(\bxi) \dxi \right]} \dx \right)}^{-\frac 1 2}
   = C {\left(\int_\Omega \rho(\bx)  d(\bx) \dx \right)}^{-1}
   \geq C {\left(\sum_K \Krho d_K \right)}^{-1} ,
\end{align*}
which gives~\eqref{eq:lambda:min:Delta} for $d=1$.

In 2D, from Lemma~\ref{lem:Green} we have
\begin{equation*}
\int_\Omega \rho(\bxi) {G(\bx, \bxi)}^2 \dxi 
 \le C \int_\Omega \rho(\bxi) \ln^2 \left( 1 + \frac{ d(\bx) d(\bxi)}{\Abs{\bx-\bxi}^2 } \right) \dxi
 \le C \int_\Omega \rho(\bxi) \ln^2 \left( 1 + \frac{ d(\bx) h_\Omega }{\Abs{\bx-\bxi}^2 } \right) \dxi ,
\end{equation*}
where $h_\Omega$ denotes the diameter of $\Omega$.
Let $B_\varepsilon (\bx )$ be the disk with center $\bx$ and radius $\varepsilon$.
Then, using property~\eqref{eq:rhoConditions} we have
\begin{align*}
   \int_\Omega  \rho(\bxi) {G(\bx, \bxi)}^2 \dxi 
      &\le C \int_{\Omega\backslash B_\varepsilon (\bx )}  
         \rho(\bxi) \ln^2 \left( 1 + \frac{ d(\bx) h_\Omega }{\varepsilon^2 } \right) \dxi 
       + C \int_{B_\varepsilon (\bx )}  
            \rho_{\max} \ln^2 \left( 1 + \frac{ d(\bx) h_\Omega }{\Abs{\bx-\bxi}^2 } \right) \dxi
      \\
      &\le C \ln^2 \left( 1 + \frac{ d(\bx) h_\Omega }{\varepsilon^2 } \right)
         + C \rho_{\max}  \int_0^{\varepsilon} r 
            \ln^2 \left( 1 + \frac{ d(\bx) h_\Omega }{r^2 } \right) d r 
      \\
      &\le C \ln^2 \left( 1 + \frac{ d(\bx) h_\Omega }{\varepsilon^2 } \right)
         + C \rho_{\max}  \varepsilon^2  
         \left[ 1 + \ln^2 \left( 1 + \frac{ d(\bx) h_\Omega }{\varepsilon^2 } \right)\right] .
\end{align*}
Taking $\varepsilon = \rho_{\max}^{-\frac 1 2}$, we get
\[
\int_\Omega \rho(\bxi)  {G(\bx, \bxi)}^2 \dxi \le C \left[ 1 + \ln^2 \bigl( 1 + d(\bx) h_\Omega \rho_{\max} \bigr)\right] .
\]
Inserting this into~\eqref{green-2}, using~\eqref{eq:rhoConditions} and the fact that $d(\bx) \le d_K$ for all $\bx \in K$, we have
\[
\lambda_\rho  \ge C {\left( 1 + \sum_{K} \Krho  \ln^2 \bigl( 1 + d_K h_\Omega \rho_{\max}\bigr)
\right)}^{-\frac 1 2} .
\]
This gives~\eqref{eq:lambda:min:Delta} for the 2D case upon absorbing $h_\Omega$ in the generic constant $C$.

For $d\geq3$, we need a slightly different version of the estimate~\eqref{green-2}.
In this case, Lemma~\ref{lem:Green} implies that $G(\bx, \cdot ) \in L^p(\Omega)$ for any $\bx \in \Omega$ and
$p \in \left(1, \frac{d}{d-2}\right)$. Using Hölder's inequality,
\[
   u_\rho(\bx) \leq \lambda_\rho 
         {\left( \int_\Omega {G(\bx,\bxi)}^p \dxi  \right)}^{\frac{1}{p}}
         {\left( \int_\Omega {\rho(\bxi)}^q {u_\rho(\bxi)}^q \dxi \right)}^{\frac{1}{q}},
\]
where $q$ satisfies $\frac{1}{p} + \frac{1}{q} = 1$.
By multiplying with $\rho(\bx)$, taking power $q$, and integrating both sides over $\Omega$, we get
\[
   \int_\Omega {\rho(\bx)}^q {u_\rho(\bx)}^q \dx 
      \leq \lambda_\rho^q
         \int_\Omega {\rho(\bx)}^q {\left(\int_\Omega {G(\bx,\bxi)}^p \dxi \right)}^{\frac{q}{p}}\dx 
         \int_\Omega {\rho(\bxi)}^q {u_\rho(\bxi)}^q \dxi 
\]
and therefore
\begin{equation}
   \lambda_\rho
      \geq {\left( \int_\Omega {\rho(\bx)}^q {\left(\int_\Omega {G(\bx,\bxi)}^p \dxi \right)}^{\frac{q}{p}}
       \dx \right)}^{-\frac{1}{q}}.
   \label{green-3}
\end{equation}

With Lemma~\ref{lem:Green} we have, using a similar strategy as in 2D,
\begin{align*}
  \int_\Omega {G(\bx,\bxi)}^p \dxi
      & \le C \int_\Omega \Abs{\bx-\bxi}^{(2-d)p} 
            \min \left\{ 1, \frac{ {d(\bx)}^p h_\Omega^p}{\Abs{\bx-\bxi}^{2p} } \right\} \dxi 
  \\
&\le C \int_{\Omega \cap B_\varepsilon(\bx)}  \Abs{\bx-\bxi}^{(2-d)p} \dxi 
         + C \int_{\Omega \setminus B_\varepsilon(\bx)} \Abs{\bx-\bxi}^{(2-d)p}
               \frac{ {d(\bx)}^p h_\Omega^p}{\Abs{\bx-\bxi}^{2p} } \dxi
\\
&\leq C \int_0^\varepsilon  r^{d-1} r^{(2-d)p} \ dr 
         + C {d(\bx)}^p h_{\Omega}^p \, \varepsilon^{-d p} 
 \\
& \le C \frac{1}{d-(d-2)p} \, \varepsilon^{d-(d-2)p}
         + C {d(\bx)}^p h_{\Omega}^p \, \varepsilon^{- d p} .
\end{align*}
Choosing $\varepsilon$ such that the terms on the right-hand side are equal, we get
\[
   \varepsilon = {\left( (d-(d-2)p) \frac{}{} {d(\bx)}^p h_\Omega^p \right)}^{\frac{1}{d+2p}} 
\]
and
\[
   \int_\Omega {G(\bx,\bxi)}^p \dxi 
      \le C {(d-(d-2)p)}^{-\frac{d p}{d+2 p}} {d(\bx)}^{p \cdot \frac{d-(d-2)p}{d + 2p}} ,
\]
where $h_{\Omega}$ has been absorbed into the generic constant $C$.
Inserting this into~\eqref{green-3} leads to
\begin{align*}
   \lambda_\rho
      &\geq C {\left( \int_\Omega {\rho(\bx)}^q {(d-(d-2)p)}^{-\frac{d q}{d+2p}} 
         {d(\bx)}^{\frac{q (d-(d-2)p)}{d+2p}}
          \dx \right)}^{-\frac{1}{q}} \\
      &\geq C {(d-(d-2)p)}^{\frac{d}{d+2 p}} 
         {\left( \int_\Omega {\rho(\bx)}^q {d(\bx)}^{ \frac{q(d-(d-2)p)}{d+2p}}
          \dx \right)}^{-\frac{1}{q}}\\
   &\geq C {(d-(d-2)p)}^{\frac{d}{d+ 2 p}} 
         {\left( \sum\limits_K \rho_K^q \Abs{K} d_K^{ \frac{q(d-(d-2)p)}{d+2p}}
            \right)}^{-\frac{1}{q}} .
\end{align*}
Since $q = p/(p-1)$ and $\rho$ is element-wise constant, this gives~\eqref{eq:lambda:min:Delta} for $d\geq3$.
\qquad
\end{proof}

\subsection{Smallest eigenvalue of the stiffness matrix}
Having established a lower bound on $\lambda_\rho$, we can now proceed with the estimation of $\lambda_{\min}(A)$ and $\lambda_{\min}(\SAS)$.
Combining~\eqref{eq:lminA:general:2} with Lemma~\ref{lem:lambda:min:Delta}, we have
\begin{equation}
   \lambda_{\min}( A ) 
      \ge C d_{\min} \WrhoMin 
   \times 
      \begin{cases}
          {\left( \sum\limits_K \Krho d_K\right)}^{-1} ,  & \text{for $d = 1$}, \\
        {\left( 1 + \sum\limits_{K} \Krho \ln^2 \left( 1 +  d_K \rho_{\max} \right)
	\right)}^{- \frac 1 2}, & \text{for $d = 2$}, \\
         {\left(\frac{d}{d-2}-p\right)}^{\frac{d}{d+2 p}} 
         {\left( \sum\limits_K \Krho^{\frac{p}{p-1}} \Abs{K}^{-\frac{1}{p-1}} d_K^{\frac{p}{p-1}\cdot\frac{d-(d-2)p}{d+2p}} \right)}^{-\frac{p-1}{p}},
         & \text{for $d\geq3$} . \\
      \end{cases}
\label{eq:A:min:i}
\end{equation}
The density function $\rho$ is arbitrary so far.
The optimal $\rho$ is such that the right-hand side term of~\eqref{eq:A:min:i} attains the maximum value. 
It is difficult, if not impossible, to find the optimal $\rho$ in general.
We follow Fried~\cite{Fri73} to choose $\rho$ such that
\begin{equation}
   \Krho = const., \quad \forall K \in \cT_h .
\label{eq:rho:i:i}
\end{equation}
This gives
\begin{equation}
   \rho(\bx) \vert_K = \frac{1}{N \Abs{K}}
   \label{rho-1}
\end{equation}
and
\begin{equation}
     \Krho = \frac{1}{N} 
     \qquad \text{and} \qquad
     \WrhoMin \geq p_{\min} \Abs{K_{\min}}_\rho = \frac{p_{\min}}{N} \geq \frac{1}{N},
     \label{eq:wrho}
\end{equation}
where $p_{\min}$ is the minimum number of elements in a mesh patch.
Inserting this into~\eqref{eq:A:min:i}, using $\abs{\bar{K}} = \Abs{\Omega} / N$ (the average element volume) and $\rho_{\max} = 1/(N \Abs{K_{\min}})$, we arrive at the following estimate.

\begin{lemma}
\label{lem:lMinA}
The smallest eigenvalue of the stiffness matrix $A$ for the linear finite element approximation of BVP~\eqref{eq:bvp1} is bounded from below by
\begin{equation}
   \lambda_{\min}(A) 
      \geq \frac{C d_{\min}}{N} 
      \times
      \begin{cases}
          {\left( \frac{1}{N} \sum\limits_K d_K \right)}^{-1} ,  & \text{for $d = 1$}, \\
        {\left( 1 + \frac{1}{N} \sum\limits_{K} \ln^2 \left( 1 +  \frac{\abs{\bar{K}}}{\Abs{K_{\min}}} d_K  \right)	\right)}^{- \frac 1 2}, & \text{for $d = 2$}, \\
         {\left(\frac{d}{d-2} - p\right)}^{\frac{d}{d+2 p}} 
         {\left( \frac{1}{N} \sum\limits_K 
            {\left(\frac{\abs{\bar{K}}}{\Abs{K}} \right)}^{\frac{1}{p-1}} d_K^{\frac{p}{p-1}\cdot\frac{d-(d-2)p}{d+2p}}
         \right)}^{- \frac{p-1}{p}} , & \text{for $d\geq3$}.\\
      \end{cases}
   \label{eq:lMin:A}
\end{equation}
\end{lemma}

\begin{rem}[Optimality of the density function]
For $d\neq2$, the choice of $\rho$ in~\eqref{rho-1} is at least optimal for element-wise constant density functions.
For $d=1$, from~\eqref{eq:A:min:i} we have 
\[
\lambda_{\min}( A )  
   \geq C d_{\min}\frac{\WrhoMin}
               {\sum\limits_K \Krho d_K }
   \geq C d_{\min}
      \frac{ p_{\min} \KrhoMin }
       { \sum\limits_K \Krho d_K }
   =  C d_{\min} 
      \frac{ p_{\min}}
      { \sum\limits_K \frac{\Krho}{\KrhoMin} d_K } .
\]
Maximizing the bound on the right-hand side is equivalent to minimizing the sum $\sum_K \frac{\Krho}{\KrhoMin} d_K $.
Since $\frac{\Krho}{\KrhoMin} \geq 1$ and $d_K > 0$ for all $K$, the optimal choice of $\rho_K$ is such that $\frac{\Krho}{\KrhoMin} = 1$ for all $K$, which is equivalent to~\eqref{eq:rho:i:i}.
The case of $d\geq3$ is similar.
\end{rem}

For the scaled stiffness matrix we have the following result.
\begin{lemma}
\label{lem:lMinSAS}
The smallest eigenvalue of the Jacobi preconditioned stiffness matrix $\SAS$ is bounded from below by
\begin{equation}
    \lambda_{\min}(\SAS )
      \geq \frac{C}{N^{\frac{2}{d}}} 
      \times
      \begin{cases}
          {\left( \frac{1}{N^2} \sum\limits_K \Abs{K} \beta_K d_K \right)}^{-1} ,
            & \text{for $d = 1$}, \\
        {\left( \frac{1}{N} \sum\limits_K \Abs{K} \beta_K \right)}^{-\frac{1}{ 2}} 
        {\left( \frac{1}{N} \sum\limits_{K} \Abs{K} \beta_K 
            \left[ 1 + \ln^2 \left( 1 +  d_K\gamma_h \right)   \right]
	         \right)}^{- \frac 1 2},
         & \text{for $d = 2$}, \\
         {\left(\frac{d}{d-2} - p \right)}^{\frac{d}{d+ 2 p}} 
         {\left( \frac{1}{N^{\frac{2 p}{d(p-1)}}} \sum\limits_K 
            \Abs{K} \beta_K^{\frac{p}{p-1}} d_K^{\frac{p}{p-1}\cdot\frac{d-(d-2)p}{d+2p}}
            \right)}^{- \frac{p-1}{p}} ,
         & \text{for $d\geq3$},\\
      \end{cases}
   \label{eq:lambda:min:SAS}
\end{equation}
where $\beta_K$ and $\gamma_K$, defined as
\begin{align}
   \beta_K & = \frac{1}{d_{\min}} \Norm{\FDF}_2 ,\quad \forall K \in \cT_h ,
      \label{beta-1}\\
   \gamma_h & = \frac{\max_K \beta_K }{\sum\limits_K \Abs{K} \beta_K} ,
      \label{gamma-1}
\end{align}
reflect the non-uniformity of the mesh viewed in the metric specified by $\D^{-1}$ (see also Remark~\ref{rem:Duniformity}).
\end{lemma}
\begin{proof}
As for the non-scaled case we have (cf.~\eqref{eq:lminA:general} and~\eqref{eq:lminB})
\[
   \lambda_{\min}(S^{-1} B_\rho S^{-1}) \geq C \min_j \frac{\Wrho}{A_{jj}}
\]
and therefore
\begin{equation}
   \lambda_{\min} (\SAS) \geq
   C d_{\min} \lambda_\rho
   \min_{j} \frac{\Wrho}{A_{jj}}.
\label{eq:lminSAS:general:2}
\end{equation}
It is known~\cite{KamHuaXu13} that
\begin{equation}
   A_{jj} \le C \sum_{K \in \omega_j}\Abs{K} \cdot \Norm{\FDF}_2 
    = C d_{\min} \sum_{K \in \omega_j} \Abs{K} \beta_K,
    \quad j = 1,\dotsc, N_{vi} ,
   \label{A-diag-1}
\end{equation}
where $C$ is a constant depending only on $\hat{K}$ and the linear basis functions on $\hat{K}$.
Thus,
\[
  \min\limits_{j} \frac{\Wrho}{A_{jj}}   
   \geq \frac{C}{ d_{\min}} \min\limits_{j} \frac{\sum_{K\in\omega_j} \Abs{K} \rho_K}
   {\sum_{K \in \omega_j} \Abs{K} \beta_K} .
\]
Once again, we choose $\rho$ to get rid of the minimum sign\footnote{Although it is not quite clear if this choice is optimal, it is the best choice we could find so far.}, viz.,
\[
   \min\limits_{j} \frac{\sum_{K\in\omega_j} \Abs{K} \rho_K}
   {\sum_{K \in \omega_j} \Abs{K} \beta_K} = const.
\]
This and~\eqref{eq:rhoConditions} lead to
\[
   \rho(\bx) \vert_K 
      = \frac{\beta_K}{\sum\limits_{\tilde{K}\in\cT_h} \abs{\tilde{K}} \beta_{\tilde{K}} }
\]
and 
\[
   \min\limits_{j} \frac{\sum_{K\in\omega_j} \Abs{K} \rho_K}
   {\sum_{K \in \omega_j} \Abs{K} \beta_K} = \frac{1}{\sum\limits_{\tilde{K}\in\cT_h} \abs{\tilde{K}} \beta_{\tilde{K}}}.
\]
The estimate~\eqref{eq:lambda:min:SAS} follows from this, inequality~\eqref{eq:lminSAS:general:2} and Lemma~\ref{lem:lambda:min:Delta}.
\qquad
\end{proof}

\section{Condition numbers of the stiffness matrix and the diagonally scaled
   stiffness matrix}\label{sect:condition}

We first quote an estimate on the maximum eigenvalues of $A$ and  $\SAS$ from~\cite{KamHuaXu13}.

\begin{lemma}[\cite{KamHuaXu13}]
\label{lem:lMaxA}
The largest eigenvalues of the stiffness matrix $A$ and preconditioned stiffness matrix $\SAS$ with the Jacobi preconditioner $S = \diag \left(  \sqrt{A_{jj}} \right) $ for the linear finite element approximation of BVP~\eqref{eq:bvp1} are bounded by
\begin{align}
   \max\limits_{j} A_{jj} &\leq \lambda_{\max}(A)  \leq (d+1) \max\limits_{j} A_{jj} 
      \leq  d_{\min}  \sqrt{d(d+3)}\max_j \sum_{k\in\omega_j} \Abs{K} \beta_K,
      \label{eq:lambdaMaxA} \\
   1 &\leq \lambda_{\max}(\SAS ) \leq d + 1.
   \label{eq:lambdaMaxScaled}
\end{align}
\end{lemma}

Combining Lemma~\ref{lem:lMaxA} with Lemmas~\ref{lem:lMinA} and~\ref{lem:lMinSAS} we obtain our main theorem.

\begin{theorem}
\label{thm:condiiton:number}
Denote the average element volume by $\abs{\bar{K}} = \frac{\Omega}{N}$ and let $p \in (1, \frac{d}{d-2})$.
Then the condition numbers of $A$  and $\SAS$ for the linear finite element approximation of BVP~\eqref{eq:bvp1} are bounded by
\begin{multline}
   \kappa(A)  
      \leq C N^{\frac 2 d} \biggl(N^{\frac{d-2}{d}}
         \max\limits_j \sum\limits_{K\in \omega_j} \Abs{K} \beta_K \biggr)
      \\
      \times
      \begin{cases}
         \frac{1}{N} \sum\limits_K d_K ,
            & \text{for $d = 1$}, \\
         {\left( 1 + \frac{1}{N} \sum\limits_{K} \ln^2 
            \left( 1 + \frac{\abs{\bar{K}}}{\Abs{K_{\min}}}  d_K \right) \right)}^{ \frac{1}{2}},
            & \text{for $d = 2$}, \\
         \frac{1}{ {\left(\frac{d}{d-2} - p \right)}^{\frac{d}{d+ 2 p}}}
            {\left( \frac{1}{N} \sum\limits_K 
            {\left(\frac{\abs{\bar{K}}}{\Abs{K}} \right)}^{\frac{1}{p-1}}
               d_K^{\frac{p}{p-1}\cdot\frac{d-(d-2)p}{d+2p}}
            \right)}^{\frac{p-1}{p}},
            & \text{for $d\geq3$},\\
      \end{cases}
   \label{eq:kappa:A}
\end{multline}
and
\begin{equation}
    \kappa(\SAS )
      \leq C N^{\frac{2}{d}} 
      \times
      \begin{cases}
         \frac{1}{N^2 } \sum\limits_K \Abs{K} \beta_K d_K ,
            & \text{for $d = 1$}, \\
         {\left( \frac{1}{N} \sum\limits_K \Abs{K} \beta_K \right)}^{\frac 1 2}
            {\left( \frac{1}{N} \sum\limits_{K} \Abs{K} \beta_K 
            \left[ 1 + \ln^2 \left( 1 +  d_K\gamma_h \right)\right]
	            \right)}^{\frac 1 2},
            & \text{for $d = 2$}, \\
         \frac{1}{ {\left(\frac{d}{d-2} - p \right)}^{\frac{d}{d+ 2 p}} } 
            {\left( \frac{1}{N^{\frac{2 p}{d (p-1)}}} 
               \sum\limits_K \Abs{K} \beta_K^{\frac{p}{p-1}}
               d_K^{\frac{p}{p-1}\cdot\frac{d-(d-2)p}{d+2p}}
            \right)}^{ \frac{p-1}{p}} ,
            & \text{for $d\geq3$},
      \end{cases}
   \label{eq:kappa:SAS}
\end{equation}
where $\beta_K$ and $\gamma_h$ defined in~\eqref{beta-1} and~\eqref{gamma-1}.
\end{theorem}

\begin{rem}
\label{rem:boundary}
From the theorem we can see that both bounds for $\kappa(A)$ and $\kappa(\SAS)$ contain the maximum distance $d_K$ from element $K$ to the boundary of the domain.
Since $d_K$ becomes smaller when $K$ is closer to $\partial \Omega$, elements close to the boundary have less influence on the bounds than those away from the boundary.
This information is useful for problems having boundary layers for which adaptive meshes are typically dense near the boundary.
\end{rem}

\begin{rem}[Comparison to~\cite{KamHuaXu13}]
The bounds in~\eqref{eq:kappa:A} and~\eqref{eq:kappa:SAS} are sharper than (in the cases of $d\le 2$) or comparable to (in the case of $d \ge 3$) those obtained in~\cite{KamHuaXu13}[Theorem~5.2] using Sobolev's inequality,
\begin{equation}
   \kappa(A) \leq C N^{\frac{2}{d}}
      \times \biggl(N^{\frac{d-2}{d}} 
       \max_j \sum\limits_{K\in\omega_j}
       \Abs{K} \beta_K \biggl)
   \times 
   \begin{cases}
      1, & \text{for $d = 1$}, \\
      1 +  \ln \frac{\abs{\bar{K}}}{\Abs{K_{\min}}}, & \text{for $d = 2$}, \\
      {\left(\frac{1}{N}\sum_{K} 
         {\left(\frac{\abs{\bar{K}}}{\Abs{K}}\right)}^{\frac{d-2}{2}} \right)}^{\frac{2}{d}},
         & \text{for $d\geq3$}.
   \end{cases}
   \label{eq:KamHuaXu13}
\end{equation}
and
\begin{equation}
 \kappa(\SAS) 
   \leq C N^{\frac{2}{d}} 
   \times 
   \begin{cases}
      \left( \frac{1}{N^2} \sum_K \Abs{K} \beta_K \right),  & \text{for $d = 1$}, \\
       \left( \frac{1}{N} \sum_K \Abs{K} \beta_K \right) (1 + \Abs{\ln \gamma_h}) ,
         & \text{for $d = 2$},  \\
      {\left( \frac{1}{N} \sum_K \Abs{K} \beta_K^{\frac{d}{2}} \right)}^{\frac{2}{d}},
         & \text{for $d\geq3$}.
   \end{cases}
   \label{eq:KamHuaXu13:SAS}
\end{equation}
For $d \le 2$, bounds~\eqref{eq:KamHuaXu13} and~\eqref{eq:KamHuaXu13:SAS} follow directly from Theorem~\ref{thm:condiiton:number} if we replace $d_K$ by its largest possible value $h_\Omega$ (the diameter of $\Omega$).
For $d\geq3$, if we replace $d_K$ with $h_{\Omega}$ and take $p$ close to $d/(d-2)$, the new bounds are very close to the bound~\eqref{eq:KamHuaXu13} and~\eqref{eq:KamHuaXu13:SAS}.
In this sense, they are comparable.
\end{rem}

\begin{rem}[$\D^{-1}$-uniform meshes]
\label{rem:Duniformity}
Like~\eqref{eq:KamHuaXu13} and~\eqref{eq:KamHuaXu13:SAS}, bounds~\eqref{eq:kappa:A} and~\eqref{eq:kappa:SAS} involve three factors: the basic factor $N^{\frac{2}{d}}$ which corresponds to the condition number of the stiffness matrix for the Laplace operator on a uniform mesh, the factor involving $\beta_K$ which reflects the mesh non-uniformity in the metric specified by $\D^{-1}$, and the factor involving the element volume which represents the mesh non-uniformity in the Euclidean metric. 
Since \[
   \frac{1}{d}\tr\left( \FDF \right) \leq \Norm{\FDF}_2 \leq \tr\left( \FDF \right),
\]
we can estimate $\beta_K$ as
\begin{multline*}
   {\left(\frac{N}{\sigma_h}\right)}^{\frac{2}{d}}
      \left(\frac{\frac{1}{d}\tr\left(\FDF\right) }{ {\det\left(\FDF\right)}^{\frac 1 d}}
      \right)
      {\left(\frac{\frac{\sigma_h}{N}}{\Abs{K} {\det(\D_K)}^{-\frac 1 2}}\right)}^{\frac 2 d}
   \\
   \le \beta_K \le d {\left(\frac{N}{\sigma_h}\right)}^{\frac{2}{d}}
      \left(\frac{\frac{1}{d}\tr\left(\FDF\right) }
         { {\det\left(\FDF\right)}^{\frac 1 d}} \right) 
      {\left(\frac{\frac{\sigma_h}{N}}{\Abs{K} {\det(\D_K)}^{-\frac 1 2}} 
         \right)}^{\frac 2 d} ,
\end{multline*}
where $\sigma_h = \sum\limits_K \Abs{K} {\det(\D_K)}^{-\frac 1 2}$.
For any uniform mesh in the metric specified by $\D^{-1}$ we have (e.g.,\ see Huang and Russell~\cite{HuaRus11})
\[
\frac{1}{d}\tr\left(\FDF\right) = {\det\left(\FDF\right)}^{\frac 1 d}
\quad\text{and}\quad \frac{\sigma_h}{N} = \Abs{K} {\det(\D_K)}^{-\frac 1 2} 
\]
and therefore
\[
   {\left(\frac{N}{\sigma_h}\right)}^{\frac{2}{d}} 
      \le \beta_K \le d {\left(\frac{N}{\sigma_h}\right)}^{\frac{2}{d}} .
\]
From~\eqref{eq:kappa:A} and~\eqref{eq:kappa:SAS}, we get
\begin{equation}
    \kappa(A) \leq C {\left(\frac{N}{\sigma_h}\right)}^{\frac 2 d} \left(N \Abs{\omega_{\max}}\right)
    \times
      \begin{cases}
         \frac{1}{N} \sum\limits_K d_K ,  & \text{for $d = 1$}, \\
         {\left( \frac{1}{N} \sum\limits_{K}
            \left[ 1 + \ln^2 \left( 1 + d_K \frac{\Abs{\bar{K}}}{\Abs{K_{\min}}} \right) \right]
	         \right)}^{ \frac 1 2}, & \text{for $d = 2$}, \\
        \frac{1}{ {\left(\frac{d}{d-2} - p  \right)}^{\frac{d}{d+ 2 p}}}
         {\left( \frac{1}{N} \sum\limits_K 
            {\left(\frac{\abs{\bar{K}}}{\Abs{K}} \right)}^{\frac{1}{p-1}}
               d_K^{\frac{p}{p-1}\cdot\frac{d-(d-2)p}{d+2p}}
         \right)}^{\frac{p-1}{p}} , & \text{for $d\geq3$},
      \end{cases}
   \label{eq:kappa:A-2}
\end{equation}
and
\begin{equation}
   \kappa(\SAS )
    \leq C {\left(\frac{N}{\sigma_h}\right)}^{\frac 2 d}
    \times
      \begin{cases}
         \sum\limits_K \Abs{K} d_K,  & \text{for $d = 1$}, \\
         {\left( \sum\limits_{K} \Abs{K} \left[ 1 + \ln^2 \left( 1 +  d_K \right)\right]
	         \right)}^{\frac 1 2}, & \text{for $d = 2$}, \\
        \frac{1}{ {\left(\frac{d}{d-2} - p  \right)}^{\frac{d}{d+ 2 p}} } 
         {\left( \sum\limits_K \Abs{K} d_K^{\frac{p}{p-1}\cdot\frac{d-(d-2)p}{d+2p}}
            \right)}^{ \frac{p-1}{p}} , & \text{for $d\geq3$}.
      \end{cases}
   \label{eq:kappa:SAS-2}
\end{equation}
If we replace $d_K$ by $h_{\Omega}$ in~\eqref{eq:kappa:SAS-2}, we get
\[
\kappa(\SAS ) \leq C {\left(\frac{N}{\sigma_h}\right)}^{\frac 2 d} ,
\]
which gives the same bound obtained in~\cite{KamHuaXu13} for $\D^{-1}$-uniform meshes.
\end{rem}

\begin{rem}
As mentioned before, the approach we used here to estimate $\lambda_{\min}(A)$ and $\lambda_{\min}(\SAS )$ was first proposed by Fried~\cite{Fri73}.
However, there is significant difference between our development and Fried's.
First, Fried obtains a bound on $\lambda_{\min}^\rho(-\Delta)$ for balls where the analytical formula of the Green's function is available and claims using a physical intuition that the bound is also valid for other domains.
Our derivation for the bound on $\lambda_{\min}^\rho(-\Delta)$ is mathematically rigorous, as seen in Lemmas~\ref{lem:1}--\ref{lem:lambda:min:Delta}.
Second, the lower bound obtained in Fried~\cite{Fri73} for $\lambda_{\min}(A)$ can be expressed in the current notation as
\begin{align}
   \lambda_{\min}(A) & \geq \frac{C d_{\min}}{N}
      \begin{cases}
         1 ,  & \text{for $d = 1$}, \\
         {\left( 1 + \ln \frac{\Abs{\bar{K}}}{\Abs{K_{\min}}}  \right)}^{-1},
	& \text{for $d = 2$}, \\ 
            {\left(\frac{\abs{\bar{K}}}{\Abs{K_{\min}}} \right)}^{\frac{2}{d}-1} , & \text{for $d\geq3$} .\\
      \end{cases}
   \label{eq:lMin:A-Fried}
\end{align}
It is easy to see that our estimate~\eqref{eq:lMin:A} is sharper than~\eqref{eq:lMin:A-Fried}.
Third, Fried estimates the maximum eigenvalue of the stiffness matrix by the maximum of the maximum eigenvalues of the local stiffness matrices, with the latter being computed directly for the Laplace operator.
On the other hand, our estimate on the maximum eigenvalue of the stiffness matrix in Lemma~\ref{lem:lMaxA} is not only more general but also more accurate.
Finally, we would like to mention that Fried~\cite{Fri73} does not study preconditioning for the stiffness matrix.
\end{rem}

\section{Numerical experiments}\label{sect:num}
The dependence of the conditioning on the diffusion matrix was extensively discussed in~\cite{KamHuaXu13} and the purpose of this paper is to investigate the dependence of the conditioning on the mesh density throughout the domain.
Thus, for the simplicity, we set $\D = I$ (Laplace operator) in all of our numerical examples.

Bounds on the smallest eigenvalue contain a constant $C$ independent of the mesh but dependent on the dimension, domain, reference element, and basis functions on the reference element.
In our computation we obtain its value empirically by calibrating the bounds for $\lambda_{\min}(\SAS)$ through comparing the exact and estimated values for a series of uniform meshes.

\subsection{1D}\label{sect:num:i}

In 1D, the new bounds~\eqref{eq:kappa:A} and~\eqref{eq:kappa:SAS} (with $\D = I$) become
\begin{align}
   & \kappa(A) \le C  \sum\limits_{K} d_K 
         \cdot \max\limits_{j} \sum\limits_{K\in \omega_j} \frac{1}{\Abs{K}} ,
      \label{eq:iD:new:i} \\
  &  \kappa(\SAS) \le C \sum\limits_{K} \frac{d_K}{\Abs{K}} .
   \label{eq:iD:new:ii}
\end{align}
On the other hand, the bounds~\eqref{eq:KamHuaXu13} and~\eqref{eq:KamHuaXu13:SAS} from~\cite{KamHuaXu13} reduce to
\begin{align}
   & \kappa(A) \le C N \cdot \max\limits_{j} 
      \sum\limits_{K\in \omega_j} \frac{1}{\Abs{K}} ,
      \label{eq:iD:old:i} \\
   & \kappa(\SAS) \le C \sum\limits_{K} \frac{1}{ \Abs{K}} .
   \label{eq:iD:old:ii}
\end{align}
We consider two set of meshes formed by Chebyshev nodes and nodes simulating extreme boundary layer refinement.

\subsubsection{Chebyshev nodes}
The first 1D example is a non-uniform mesh given by the end points $x_0 = 0$ and $x_N = 1$ and the Chebyshev nodes,
\begin{align}
      x_j =  \frac{1}{2} \bigl( 1 - \cos(\xi_j) \bigr),
          \quad j = 1,\dotsc, N-1,
   \label{eq:che:nodes}
\end{align}
where
\begin{align*}
\xi_0 &= 0, \ \xi_N = 1 \quad \text{and} \quad 
   \xi_j = \frac{\pi (2 j -1)}{2 \left(N-1\right)}, 
   \quad j = 1, \dotsc, N-1.
\end{align*}
For this mesh, we have
\[
   \Abs{K_j} = x_j - x_{j-1} \sim \frac{1}{N} \sin(\xi_j),
   \quad d_{K_j} = \min \left( \frac{1}{2}\bigl(1-\cos(\xi_j) \bigr), 
      \frac{1}{2} \bigl(1+ \cos(\xi_{j-1}) \bigr) \right).
\]
Using these, we get
\[
   \max\limits_{j} \sum\limits_{K\in \omega_j} \frac{1}{\Abs{K}}  
      \sim \max\limits_{j} \frac{1}{\Abs{K_j}} \sim N^2 
\]
and
\[
   \frac{1}{N} \sum\limits_{K} d_K \sim \frac{2}{N} 
      \sum\limits_{j=1}^{N/2} \frac{1}{2} \bigl(1-\cos(\xi_j) \bigr)
\sim \int_0^{\frac{\pi}{2}} \bigl(1-\cos(\xi) \bigr) ~d\xi \sim 1 .
\]
Thus, from~\eqref{eq:iD:new:i} and~\eqref{eq:iD:old:i} we see that the new bound and the bound from~\cite{KamHuaXu13} both lead to
\begin{equation*}
   \kappa(A) \le C N^3 .
\end{equation*}

For the scaled case, we have
\[
   \sum\limits_{K} \frac{d_K}{\Abs{K}} 
   \sim 2 \sum\limits_{j=1}^{N/2} 
         \frac{\frac{1}{2}\bigl(1-\cos(\xi_j)\bigr)}
         {\frac{\pi}{N}\sin(\xi_j) }
   \sim N^2 \int_0^{\frac{\pi}{2}} \frac{1-\cos(\xi)}{\sin(\xi)} \ d\xi
   \sim N^2
\]
and
\[
   \sum\limits_{K} \frac{1}{\Abs{K}} 
   \sim 2 \sum\limits_{j=1}^{N/2} \frac{1}{\frac{\pi}{N}\sin(\xi_j) }
   \sim N^2 \int_{\xi_1}^{\frac{\pi}{2}}\frac{1}{\sin(\xi)} \ d\xi
   \sim N^2 \ln \Abs{\tan(\xi/2)} \vert_{\xi_1}^{\frac{\pi}{2}} 
   \sim N^2\ln N .
\]
Thus, we have
   \begin{equation}
   \kappa(\SAS) \le C N^2
   \label{eq:iD:ii}
\end{equation}
for the new bound~\eqref{eq:iD:new:ii} and
\begin{equation*}
   \kappa(\SAS) \le C N^2 \ln N
\end{equation*}
from the bound~\eqref{eq:iD:old:ii} from~\cite{KamHuaXu13}.
Notice that~\eqref{eq:iD:ii} has the same order as $\kappa(\SAS)$ for a uniform mesh as $N$ increases.

The numerical comparison of the estimated and exact values for $\kappa(A)$ and $\kappa(\SAS)$ are presented in Fig.~\ref{fig:I:che}.
As expected, for the non-scaled case both estimates are comparable and very tight (Fig.~\ref{fig:I:che:A}).
On the other hand, the new estimate~\eqref{eq:kappa:A} is more accurate than the one from~\cite{KamHuaXu13} if diagonal scaling is applied (Fig.~\ref{fig:I:che:SAS}).
For this example, the new bound on $\kappa(\SAS)$ seems to be asymptotically exact.

\subsubsection{Boundary layer refinement}
The second 1D example simulates boundary layer mesh refinement towards the boundary point $x=0$ with the internal nodes 
\begin{equation}
   x_j = {2^j} / {2^N}, \qquad j = 1,\dotsc, N-1.
   \label{eq:power2:nodes}
\end{equation}
For this mesh, we have
\begin{align*}
   \Abs{K_j} 
   = x_j - x_{j-1}
   = \frac{2^{j-1}}{2^N}, 
      \quad  d_{K_j} = \min \left(\frac{2^{j}}{2^N}, 
         1 - \frac{2^{j-1}}{2^N} \right).
\end{align*}
Then, bounds~\eqref{eq:iD:old:i} and~\eqref{eq:iD:old:ii} for~\cite{KamHuaXu13} become
\begin{align*}
   &\kappa(A) \le C N 2^N
   \quad \text{and} \quad
   \kappa(\SAS) \le C 2^N,
\end{align*}
whereas the new bounds~\eqref{eq:iD:new:i} and~\eqref{eq:iD:new:ii} give
\begin{align*}
   \kappa(A) \le C 2^N 
   \quad \text{and} \quad
   \kappa(\SAS) \le C N.
\end{align*}
This shows that, for boundary layer refinement, the new bound for the scaled case is a significant improvement.
Note that $\kappa(\SAS) = \cO(N)$, which has a smaller order as $N \to \infty$ than the condition number for a uniform mesh (which is $\cO(N^2)$).
Thus, for 1D problems with steep boundary layers, strong mesh concentration towards the boundary not only increases the accuracy of the solution but, at the same time, improves the conditioning.

Numerical results are in perfect agreement with the analysis.
Indeed, Fig.~\ref{fig:I:power2:A} shows that both new and old estimates are comparable for the non-scaled case, with the new one being slightly more accurate.
After scaling (Fig.~\ref{fig:I:power2:SAS}), the situation is quite different: the new estimate is very close to the exact value whereas the estimate from~\cite{KamHuaXu13} exhibits a dramatic overestimation.

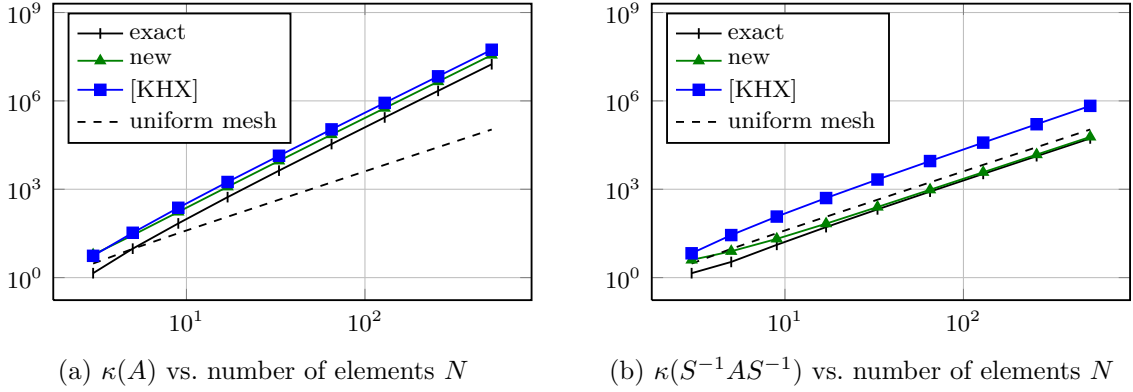
\begin{figure}[t]
   \begin{subfigure}[t]{0.5\textwidth} \centering
   \begin{tikzpicture}
      \begin{loglogaxis}[legend pos=north west, ymax=2.0e+09,]
      \addplot[color=\MyColorExact, \MyLineStyle, mark=\MyMarkExact,]  
         table [x index=0, y index=1, col sep = space] {\dataPath/chebyshev-kappa.dat};
      \addlegendentry{exact}
      \addplot[color=\MyColorRho, \MyLineStyle, mark=\MyMarkRho,]
         table [x index=0, y index=10, col sep = space] {\dataPath/chebyshev-kappa.dat};
      \addlegendentry{new}
      \addplot[color=\MyColorKHX, \MyLineStyle, mark=\MyMarkKHX]
         table [x index=0, y index=2, col sep = space] {\dataPath/chebyshev-kappa.dat};
      \addlegendentry{\cite{KamHuaXu13}}
      \addplot[color=\MyColorUniform, \MyLineStyleUniform]
         table [x index=0, y index=1, col sep = space] {\dataPath/uniform1-kappa.dat};
      \addlegendentry{uniform mesh}
      \end{loglogaxis}
   \end{tikzpicture}
   \caption{$\kappa(A)$ vs.\ number of elements $N$}\label{fig:I:che:A}
   \end{subfigure}%
   \begin{subfigure}[t]{0.5\textwidth} \centering
   \begin{tikzpicture}
      \begin{loglogaxis}[legend pos=north west, ymax=2.0e+09,]
      \addplot[color=\MyColorExact, \MyLineStyleScaled, mark=\MyMarkExact,]  
         table [x index=0, y index=4, col sep = space] {\dataPath/chebyshev-kappa.dat};
      \addlegendentry{exact}
      \addplot[color=\MyColorRho, \MyLineStyleScaled, mark=\MyMarkRho,]
         table [x index=0, y index=11, col sep = space] {\dataPath/chebyshev-kappa.dat};
      \addlegendentry{new}
      \addplot[color=\MyColorKHX, \MyLineStyleScaled, mark=\MyMarkKHX]
         table [x index=0, y index=5, col sep = space] {\dataPath/chebyshev-kappa.dat};
      \addlegendentry{\cite{KamHuaXu13}}
      \addplot[color=\MyColorUniform, \MyLineStyleUniform]
         table [x index=0, y index=1, col sep = space] {\dataPath/uniform1-kappa.dat};
      \addlegendentry{uniform mesh}
      \end{loglogaxis}
   \end{tikzpicture}
   \caption{$\kappa(\SAS)$ vs.\ number of elements $N$}\label{fig:I:che:SAS}
   \end{subfigure}
   \caption{Conditioning for the Chebyshev nodes~\eqref{eq:che:nodes}}\label{fig:I:che}
\end{figure}
\begin{figure}[t]
   \begin{subfigure}[t]{0.5\textwidth} \centering
   \begin{tikzpicture}
      \begin{loglogaxis}[legend pos=north west, ymax=4.0e+12,
         xticklabel=\pgfmathparse{exp (\tick)}\pgfmathprintnumber{\pgfmathresult},
         xtick = {3, 5, 9, 17, 33}, ]
      \addplot[color=\MyColorExact, \MyLineStyle, mark=\MyMarkExact,]  
         table [x index=0, y index=1, col sep = space] {\dataPath/power2-kappa.dat};
      \addlegendentry{exact}
      \addplot[color=\MyColorRho, \MyLineStyle, mark=\MyMarkRho,]
         table [x index=0, y index=10, col sep = space] {\dataPath/power2-kappa.dat};
      \addlegendentry{new}
      \addplot[color=\MyColorKHX, \MyLineStyle, mark=\MyMarkKHX]
         table [x index=0, y index=2, col sep = space] {\dataPath/power2-kappa.dat};
      \addlegendentry{\cite{KamHuaXu13}}
      \addplot[color=\MyColorUniform, \MyLineStyleUniform]
         table [x index=0, y index=1, col sep = space] {\dataPath/uniform1-p-kappa.dat};
      \addlegendentry{uniform mesh}
      \end{loglogaxis}
   \end{tikzpicture}
   \caption{$\kappa(A)$ vs.\ number of elements $N$}\label{fig:I:power2:A}
   \end{subfigure}%
   \begin{subfigure}[t]{0.5\textwidth} \centering
   \begin{tikzpicture}
      \begin{loglogaxis}[legend pos=north west, ymax=4.0e+12,
         xticklabel=\pgfmathparse{exp (\tick)}\pgfmathprintnumber{\pgfmathresult},
         xtick = {3, 5, 9, 17, 33}, ]
      \addplot[color=\MyColorExact, \MyLineStyleScaled, mark=\MyMarkExact,]  
         table [x index=0, y index=4, col sep = space] {\dataPath/power2-kappa.dat};
      \addlegendentry{exact}
      \addplot[color=\MyColorRho, \MyLineStyleScaled, mark=\MyMarkRho,]
         table [x index=0, y index=11, col sep = space] {\dataPath/power2-kappa.dat};
      \addlegendentry{new}
      \addplot[color=\MyColorKHX, \MyLineStyleScaled, mark=\MyMarkKHX]
         table [x index=0, y index=5, col sep = space] {\dataPath/power2-kappa.dat};
      \addlegendentry{\cite{KamHuaXu13}}
      \addplot[color=\MyColorUniform, \MyLineStyleUniform]
         table [x index=0, y index=1, col sep = space] {\dataPath/uniform1-p-kappa.dat};
      \addlegendentry{uniform mesh}
      \end{loglogaxis}
   \end{tikzpicture}
   \caption{$\kappa(\SAS)$ vs.\ number of elements $N$}\label{fig:I:power2:SAS}
   \end{subfigure}
   \caption{Conditioning for the boundary layer 
      refinement~\eqref{eq:power2:nodes}}\label{fig:I:power2}
\end{figure}
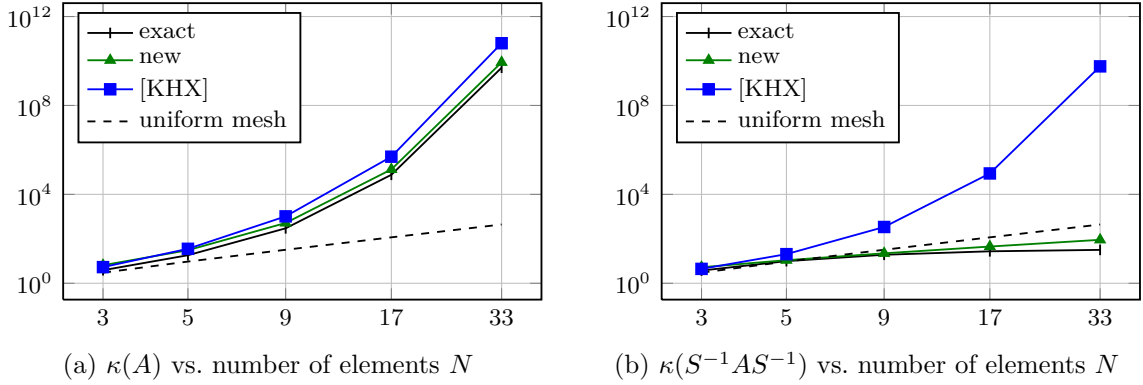

\subsection{2D}\label{sect:num:2d}

\begin{figure}[p]
   \subcaptionbox{2D meshes used in Sect.~\ref{sect:num:2d}\label{fig:II:mesh}}
   [0.5\linewidth] {
      \centering
      \includegraphics[width=0.25\linewidth, height=0.25\linewidth, clip]
         {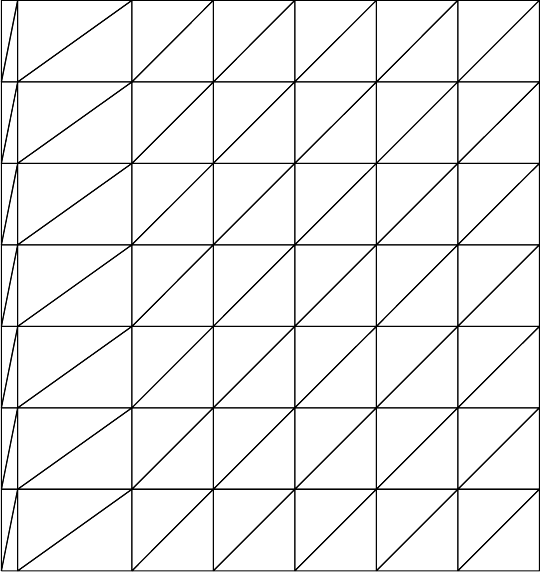}
   }%
   \subcaptionbox{3D meshes used in Sect.~\ref{sect:num:3d}\label{fig:III:mesh}}
   [0.5\linewidth] {
      \centering
      \includegraphics[width=0.25\linewidth, height=0.25\linewidth, clip]
         {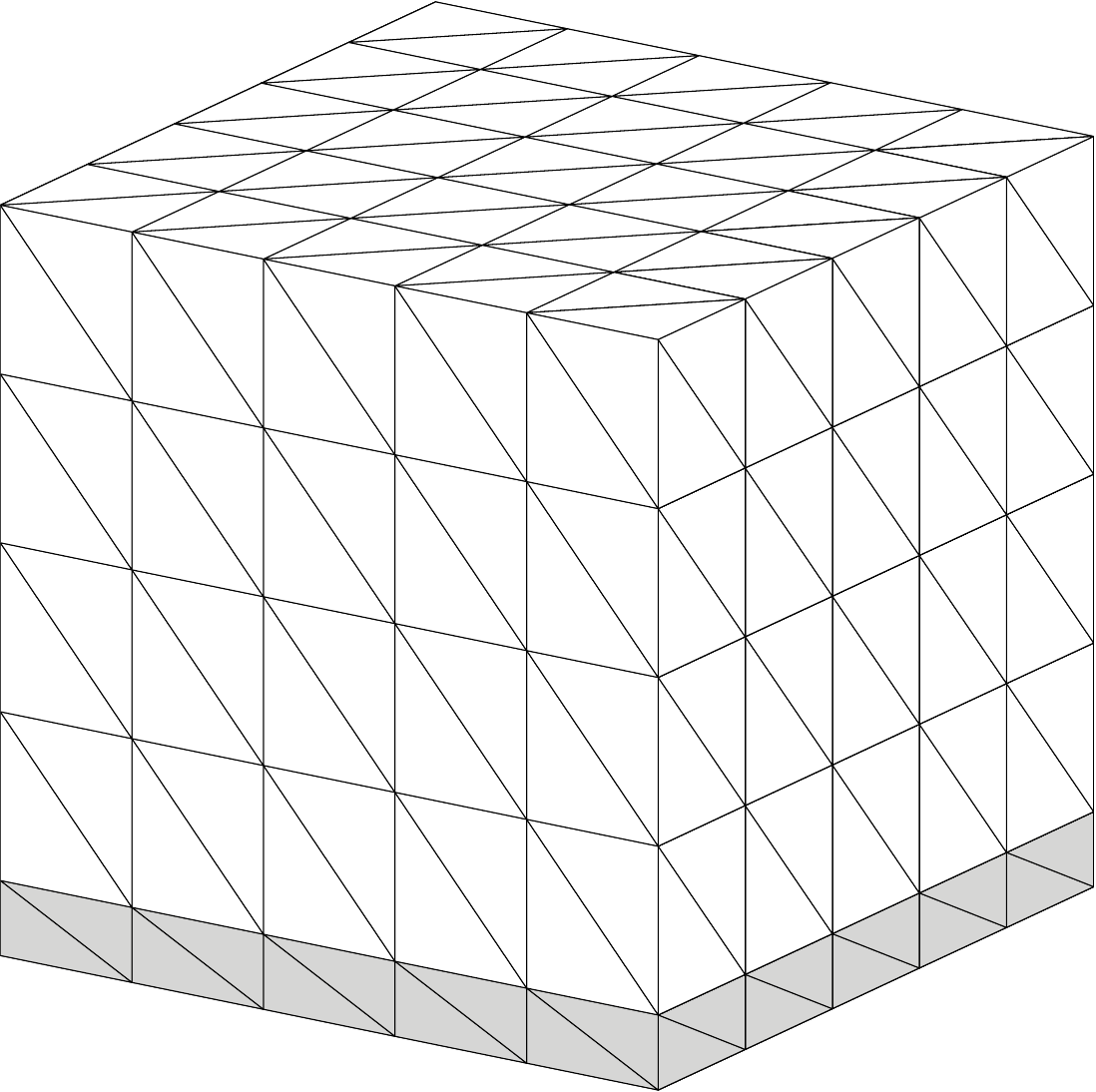}
   }
   \caption{Test meshes in two and three dimensions}\label{fig:meshes}
\end{figure}

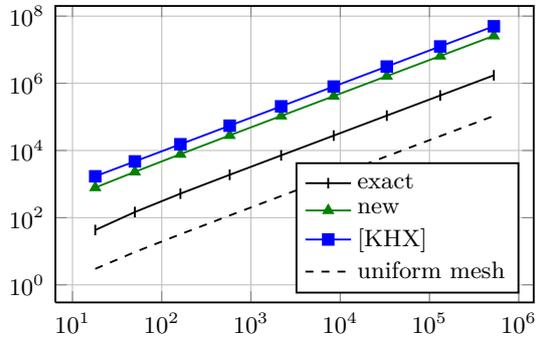
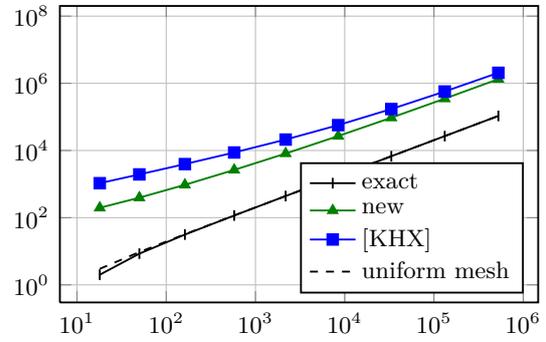
\begin{figure}[p]
   \begin{subfigure}[p]{0.5\textwidth} \centering
      \begin{tikzpicture}
      \begin{loglogaxis}[legend pos=south east, ymin=0.3, ymax=2.0e+08, ]
         \addplot[color=\MyColorExact, \MyLineStyle, mark=\MyMarkExact,]  
            table [x index=0, y index=1, col sep = space] {\dataPath/n2-bndy-kappa.dat};
         \addlegendentry{exact}
         \addplot[color=\MyColorRho, \MyLineStyle, mark=\MyMarkRho]
            table [x index=0, y index=10, col sep = space] {\dataPath/n2-bndy-kappa.dat};
         \addlegendentry{new}
         \addplot[color=\MyColorKHX, \MyLineStyle, mark=\MyMarkKHX]
            table [x index=0, y index=2, col sep = space] {\dataPath/n2-bndy-kappa.dat};
         \addlegendentry{\cite{KamHuaXu13}}
         \addplot[color=\MyColorUniform, \MyLineStyleUniform]
            table [x index=0, y index=1, col sep = space] {\dataPath/uniform2-kappa.dat};
         \addlegendentry{uniform mesh}
           \end{loglogaxis}
      \end{tikzpicture}
      \caption{$\kappa(A)$ vs.\ number of elements $N$}\label{fig:II:n:A}
   \end{subfigure}%
   \begin{subfigure}[p]{0.5\textwidth} \centering
      \begin{tikzpicture}
      \begin{loglogaxis}[legend pos=south east, ymin=0.3, ymax=2.0e+08,]
         \addplot[color=\MyColorExact, \MyLineStyleScaled, mark=\MyMarkExact,]  
            table [x index=0, y index=4, col sep = space] {\dataPath/n2-bndy-kappa.dat};
         \addlegendentry{exact}
         \addplot[color=\MyColorRho, \MyLineStyleScaled, mark=\MyMarkRho]
            table [x index=0, y index=11, col sep = space] {\dataPath/n2-bndy-kappa.dat};
         \addlegendentry{new}
         \addplot[color=\MyColorKHX, \MyLineStyleScaled, mark=\MyMarkKHX]
            table [x index=0, y index=5, col sep = space] {\dataPath/n2-bndy-kappa.dat};
         \addlegendentry{\cite{KamHuaXu13}}
         \addplot[color=\MyColorUniform, \MyLineStyleUniform]
            table [x index=0, y index=1, col sep = space] {\dataPath/uniform2-kappa.dat};
         \addlegendentry{uniform mesh}
           \end{loglogaxis}
      \end{tikzpicture}
      \caption{$\kappa(\SAS)$ vs.\ number of elements $N$}\label{fig:II:n:SAS}
   \end{subfigure}
   \caption{2D meshes with a fixed maximum aspect ratio 
      of $125:1$ and changing $N$}\label{fig:II:n}
\end{figure}

\begin{figure}[p]
   \begin{subfigure}[p]{0.5\textwidth} \centering
      \begin{tikzpicture}
      \begin{loglogaxis}[legend pos=north west, ymin=1.0e+03, ymax=2.0e+08, ]
         \addplot[color=\MyColorExact, \MyLineStyle, mark=\MyMarkExact,]  
            table [x index=8, y index=1, col sep = space] {\dataPath/ar2-bndy-kappa.dat};
         \addlegendentry{exact}
         \addplot[color=\MyColorRho, \MyLineStyle, mark=\MyMarkRho]
            table [x index=8, y index=10, col sep = space] {\dataPath/ar2-bndy-kappa.dat};
         \addlegendentry{new}
         \addplot[color=\MyColorKHX, \MyLineStyle, mark=\MyMarkKHX]
            table [x index=8, y index=2, col sep = space] {\dataPath/ar2-bndy-kappa.dat};
         \addlegendentry{\cite{KamHuaXu13}}
          \end{loglogaxis}
      \end{tikzpicture}
      \caption{$\kappa(A)$ vs.\ maximum aspect ratio}\label{fig:II:ar:A}
   \end{subfigure}%
   \begin{subfigure}[p]{0.5\textwidth} \centering
      \begin{tikzpicture}
      \begin{loglogaxis}[legend pos=north west, ymin=1.0e+03, ymax=2.0e+08,]
         \addplot[color=\MyColorExact, \MyLineStyleScaled, mark=\MyMarkExact,]  
            table [x index=8, y index=4, col sep = space] {\dataPath/ar2-bndy-kappa.dat};
         \addlegendentry{exact}
         \addplot[color=\MyColorRho, \MyLineStyleScaled, mark=\MyMarkRho]
            table [x index=8, y index=11, col sep = space] {\dataPath/ar2-bndy-kappa.dat};
         \addlegendentry{new}
         \addplot[color=\MyColorKHX, \MyLineStyleScaled, mark=\MyMarkKHX]
            table [x index=8, y index=5, col sep = space] {\dataPath/ar2-bndy-kappa.dat};
         \addlegendentry{\cite{KamHuaXu13}}
          \end{loglogaxis}
      \end{tikzpicture}
      \caption{$\kappa(\SAS)$ vs.\ maximum aspect ratio}\label{fig:II:ar:SAS}
   \end{subfigure}
   \caption{2D meshes with a fixed $N = 20\,000$ and changing aspect ratio}\label{fig:II:ar}
\end{figure}
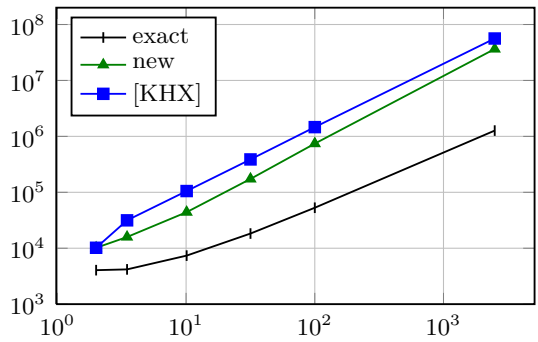
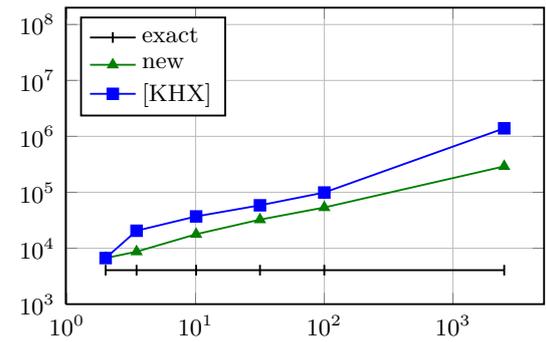

In 2D we consider a mesh for the unit square $[0,1]\times[0,1]$ with $\cO(\sqrt{N})$ skew elements near the boundary, as shown in Fig.~\ref{fig:II:mesh}.

First, we set the maximum aspect ratio at $125:1$ and verify the dependence of the condition number on the number of mesh elements $N$ (Fig.~\ref{fig:II:n}).
Then, we set $N = 20\,000$ and change the maximum aspect ratio of the mesh elements to investigate the dependence of the conditioning on the mesh shape (Fig.~\ref{fig:II:ar}).

The observation is that the new and the~\cite{KamHuaXu13} estimates are comparable, with the new one being slightly more accurate, especially when changing the maximal mesh aspect ratio.
However, neither of the estimates is as accurate as in 1D, meaning that our estimation of the smallest eigenvalue can be further improved.
This observation is essentially valid in higher dimensions as well. 
Since the Green's function has a singularity for $d\geq2$ (and depends on the shape of the domain), it is difficult to obtain estimates as sharp as in 1D in general.
Interestingly, in our example, the exact condition number of the scaled stiffness matrix appears to be independent of the aspect ratio of the boundary layer elements.
This  suggests that the optimal estimate in 2D, similar to the 1D case, should be
\begin{equation}
   \kappa(\SAS) \le C  \sum_K \Abs{K} \beta_K
         \log\left(1 + \frac{d_K}{\Abs{K}}\right).
   \label{eq:SAS:supposed}
\end{equation}

\subsection{3D}\label{sect:num:3d}
Similarly to the 2D case, we use a mesh for the unit cube ${[0,1]}^3$ with $\cO(N^{2/3})$ skew elements near the boundary (Fig.~\ref{fig:III:mesh}) and consider two different settings: fixed anisotropy ($25 : 25 : 1$) with increasing number of elements and fixed $N = 29\,478$ paired with the changing anisotropy of the mesh.
Numerical results for $p=2.9$ are presented in Figs.~\ref{fig:III:n} and~\ref{fig:III:ar}.

First, we observe that the new estimate is comparable to the estimate from~\cite{KamHuaXu13} for both non-scaled and scaled cases.
Second, as in the 2D case, we observe that increasing the maximum aspect ratio of elements at the boundary has no impact on the exact condition number of the Jacobi preconditioned stiffness matrix (at least for the considered mesh type).
This indicates that the obtained bounds for the scaled stiffness matrix are not optimal.
The numerical results also suggest that the optimal bound should have a much stronger dependence on the distance $d_K$ of an element $K$ from the boundary of the domain.

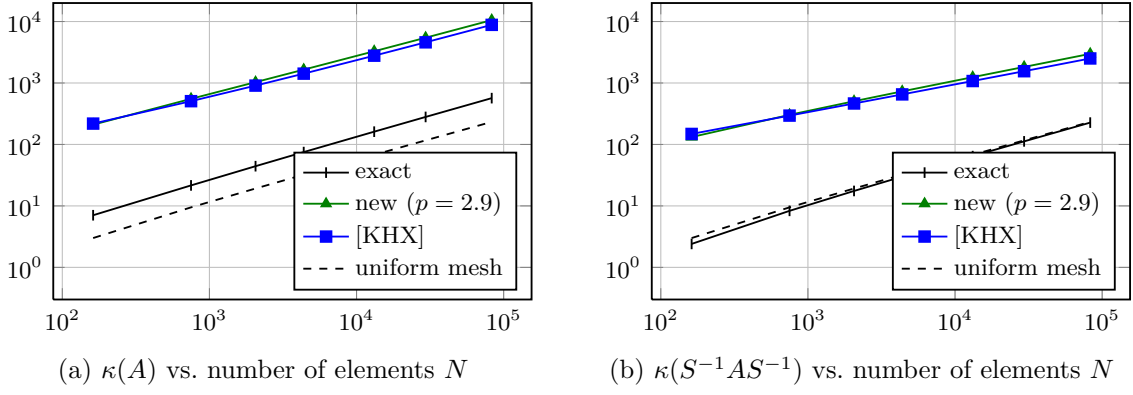
\begin{figure}[t]
  \begin{subfigure}[t]{0.5\textwidth} \centering
      \begin{tikzpicture}
      \begin{loglogaxis}[legend pos=south east, ymin=0.3e+00, ymax=2.0e+04, ]
         \addplot[color=\MyColorExact, \MyLineStyle, mark=\MyMarkExact,]  
            table [x index=0, y index=1, col sep = space] {\dataPath/n3-bndy-kappa.dat};
         \addlegendentry{exact}
         \addplot[color=\MyColorRho, \MyLineStyle, mark=\MyMarkRho]
            table [x index=0, y index=10, col sep = space] {\dataPath/n3-bndy-kappa.dat};
         \addlegendentry{new ($p=2.9$)}
         \addplot[color=\MyColorKHX, \MyLineStyle, mark=\MyMarkKHX]
            table [x index=0, y index=2, col sep = space] {\dataPath/n3-bndy-kappa.dat};
         \addlegendentry{\cite{KamHuaXu13}}
         \addplot[color=\MyColorUniform, \MyLineStyleUniform]
            table [x index=0, y index=1, col sep = space] {\dataPath/uniform3-kappa.dat};
         \addlegendentry{uniform mesh}
         \end{loglogaxis}
      \end{tikzpicture}
      \caption{$\kappa(A)$ vs.\ number of elements $N$}\label{fig:III:n:A}
   \end{subfigure}%
   \begin{subfigure}[t]{0.5\textwidth} \centering
      \begin{tikzpicture}
      \begin{loglogaxis}[legend pos=south east, ymin=0.3e+00, ymax=2.0e+04, ]
         \addplot[color=\MyColorExact, \MyLineStyleScaled, mark=\MyMarkExact,]  
            table [x index=0, y index=4, col sep = space] {\dataPath/n3-bndy-kappa.dat};
         \addlegendentry{exact}
         \addplot[color=\MyColorRho, \MyLineStyleScaled, mark=\MyMarkRho]
            table [x index=0, y index=11, col sep = space] {\dataPath/n3-bndy-kappa.dat};
         \addlegendentry{new ($p=2.9$)}
         \addplot[color=\MyColorKHX, \MyLineStyleScaled, mark=\MyMarkKHX]
            table [x index=0, y index=5, col sep = space] {\dataPath/n3-bndy-kappa.dat};
         \addlegendentry{\cite{KamHuaXu13}}
         \addplot[color=\MyColorUniform, \MyLineStyleUniform]
            table [x index=0, y index=1, col sep = space] {\dataPath/uniform3-kappa.dat};
         \addlegendentry{uniform mesh}
           \end{loglogaxis}
      \end{tikzpicture}
      \caption{$\kappa(\SAS)$ vs.\ number of elements $N$}\label{fig:III:n:SAS}
   \end{subfigure}
   \caption{3D meshes with a fixed aspect ratio 
      of $25:25:1$ and changing $N$}\label{fig:III:n}
\end{figure}

\begin{figure}[t]
   \begin{subfigure}[t]{0.5\textwidth} \centering
      \begin{tikzpicture}
      \begin{loglogaxis}[legend pos=north west, ymin=4.0e+01, ymax=2.0e+06]
         \addplot[color=\MyColorExact, \MyLineStyle, mark=\MyMarkExact,]  
            table [x index=8, y index=1, col sep = space] {\dataPath/ar3-bndy-kappa.dat};
         \addlegendentry{exact}
         \addplot[color=\MyColorRho, \MyLineStyle, mark=\MyMarkRho]
            table [x index=8, y index=10, col sep = space] {\dataPath/ar3-bndy-kappa.dat};
         \addlegendentry{new ($p=2.9$)}
         \addplot[color=\MyColorKHX, \MyLineStyle, mark=\MyMarkKHX]
            table [x index=8, y index=2, col sep = space] {\dataPath/ar3-bndy-kappa.dat};
         \addlegendentry{\cite{KamHuaXu13}}
          \end{loglogaxis}
      \end{tikzpicture}
      \caption{$\kappa(A)$ vs.\ number of elements $N$}\label{fig:III:ar:A}
   \end{subfigure}%
   \begin{subfigure}[t]{0.5\textwidth} \centering
      \begin{tikzpicture}
      \begin{loglogaxis}[legend pos=north west, ymin=4.0e+01, ymax=2.0e+06]
         \addplot[color=\MyColorExact, \MyLineStyleScaled, mark=\MyMarkExact,]  
            table [x index=8, y index=4, col sep = space] {\dataPath/ar3-bndy-kappa.dat};
         \addlegendentry{exact}
         \addplot[color=\MyColorRho, \MyLineStyleScaled, mark=\MyMarkRho]
            table [x index=8, y index=11, col sep = space] {\dataPath/ar3-bndy-kappa.dat};
         \addlegendentry{new ($p=2.9$)}
         \addplot[color=\MyColorKHX, \MyLineStyleScaled, mark=\MyMarkKHX]
            table [x index=8, y index=5, col sep = space] {\dataPath/ar3-bndy-kappa.dat};
         \addlegendentry{\cite{KamHuaXu13}}
          \end{loglogaxis}
      \end{tikzpicture}
      \caption{$\kappa(\SAS)$ vs.\ number of elements $N$}\label{fig:III:ar:SAS}
   \end{subfigure}
   \caption{3D meshes with a fixed $N=29\,478$ 
      and changing aspect ratio}\label{fig:III:ar}
\end{figure}

\section{Conclusions}\label{SEC:conclusion}

In the previous sections we have studied the linear finite element approximation of the boundary value problem~\eqref{eq:bvp1} with general nonuniform meshes and developed bounds for the condition numbers of the stiffness matrix and Jacobi preconditioned stiffness matrix.
The main result (Theorem~\ref{thm:condiiton:number}) shows that the density function approach of Fried~\cite{Fri73} can be made mathematically rigorous for general domains and lead to estimates that provide more detail and are sharper than existing estimates for general adaptive meshes in one and two dimensions and comparable in three and higher dimensions.
Moreover, the bounds in Theorem~\ref{thm:condiiton:number} involves a factor $d_K$ which describes the maximum distance from element $K$ to the boundary of the domain and becomes smaller when $K$ is closer to $\partial \Omega$.
They reveal that the mesh concentration near the boundary has less influence on the condition number than the mesh concentration in the interior of the domain.
This is especially true for the Jacobi preconditioned system where the former has little or almost no effect on the condition number.
The numerical results presented in Sect.~\ref{sect:num} confirm the theoretical analysis although they also suggest that the new bounds could be further improved in two and higher dimensions.

\printbibliography{}

\end{document}